\documentclass[a4paper,10pt]{article}
\usepackage{latexsym,amssymb,amsfonts,amsmath,amsthm,mathrsfs,amstext,color,graphicx,times}
\usepackage[mathscr]{euscript}
\usepackage{indentfirst}
\usepackage{cite}
\usepackage{wrapfig}
\usepackage[all]{xy}
\usepackage{tikz}
\usetikzlibrary{patterns}

\usepackage{enumitem}

\newcommand{\R}{\mathbb{R}}
\newcommand{\Q}{\mathbb{Q}}

\newcommand{\N}{\mathbb{N}}

\newtheorem{theorem}{Theorem}[section]
\newtheorem{coro}{Corollary}[theorem]
\newtheorem{remark}{Remark}[section]
\newtheorem{lemma}{Lemma}[section]
\newtheorem{proposition}{Proposition}[section]
\newtheorem{cor}{Corollary}[proposition]
\newtheorem{example}{Example}[section]

\usepackage{wrapfig}
\usepackage[all]{xy}
\usepackage{tikz}
\usepackage{hyperref}
\usepackage{animate}
\title{Existence results for equilibrium problem }
\author{
John Cotrina\thanks{Universidad del Pac\'ifico.
Av. Salaverry 2020, Jes\'us Mar\'ia, Lima, Per\'u. Email: \texttt{\{cotrina\_je,garcia\_yv\}@up.edu.pe}
}\and Yboon Garc\'ia\footnotemark[1]}

\begin{document}

\maketitle
\begin{abstract}
In this work, we introduce the notion of regularization of bifunctions in a similar way 
as the well-known convex, quasiconvex and lower semicontinuous regularizations due to Crouzeix.
We show that the Equilibrium Problems associated to bifunctions and their regularizations  are equivalent
in the sense of having the same solution set. Also, we present
new existence results of solutions for Equilibrium Problems.

\end{abstract}
\bigskip


\bigskip


\section{Introduction}
Given a real Banach space $X$, a nonempty subset $K$ of $X$ and a bifunction
$f:K\times K\to\R$. The \emph{Equilibrium Problem}, (EP) for short,  is defined as follows:
\begin{equation}\label{EP}
\hbox{Find } x\in K \hbox{ such that }  f(x,y)\geq0\mbox{ for all }y\in K.\tag{EP}
\end{equation}

Equilibrium Problems have been extensively studied in recent years 
(e.g., \cite{BP01,BP05,OB93,castellani2010,castellani2012,FB00,FB03,IKS06,IKS09,IS03}).
Particularly, It is well known that many problems such as
variational inequality problems, fixed-point problems, Nash equilibrium problems and optimization problems, 
among others, can be reformulated as equilibrium problems.
(see for instance \cite{OB93,IS03,SN10,Fa-Za}).

A recurrent subject  in the analysis of this problem is  the connection between the solution sets of (EP) 
and the solution set of the following problem:

\begin{equation}\label{cfp}
\mbox{Find }x\in K\mbox{ such that }f(y,x)\leq 0\mbox{ for all }y\in K. \tag{CFP}
\end{equation}
This can be seen as a dual formulation of \eqref{EP} and it corresponds to a particular case of the convex 
feasibility problem (cfr. \cite{KKM,Kfan2}).


It was proved in \cite{IS03} that if $f$ is {upper} semicontinuous in the first argument, convex 
{and lower semicontinuous} in the second one and it vanishes
on the diagonal $K\times K$, then every solution of (\ref{cfp}) is a solution of 
(\ref{EP}), and moreover both solution sets
trivially coincide under pseudomonotonicity of $f$.

In order to establish the nonemptiness of the solution set of (\ref{cfp}) and the inclusion of this set in solution set of (\ref{EP}) 
in \cite{BP05}, Bianchi and Pini introduced the concept of \emph{local convex feasibility problem} and the 
\emph{upper sign continuity} for bifunctions as an adaptation
of the set-valued map introduced in \cite{H03},  by Hadjisavvas. 
They adaptated the existence result for 
variational inequalities developed by Aussel and Hadjisavvas in \cite{AH04}. Basically, they proved 
that every solution of (\ref{cfp}) is a local solution of (\ref{cfp}) and all local solution of (\ref{CFP}) is a solution of (\ref{EP}).
Following the same way, in \cite{castellani2012}, 
{Castellani} and Giuli introduced the concept of \emph{upper sign property} for bifunction as a
local property which is weaker than the upper sign continuity and they extend the result obtained by Bianchi and Pini. 


Our aim in this paper is to provide sufficient conditions for the existence of solutions under weak assumptions on
the bifunction and some coercivity conditions. We introduce, in Section 3, the regularization of a bifunction analogously of regularization 
of functions introduced in \cite{JPC} by Crouzeix
and we study the properties of such regularization. In section 4, we establish that
the equilibrium problems associated to a bifunction and its regularization are equivalent in the sense 
{Castellani} and Giuli. (cf. \cite{castellani2010}). 
We provide, in  Section 5, sufficient conditions for 
the existence of solutions for (EP). 


\section{Preliminary definitions and notations}
Let $X$ be a real topological vector space, and let $A\subset X$. We denote by
$\overline{A}$, ${\rm co}(A)$ and $\overline{{\rm co}}(A)$ the smallest closed set, convex set and 
closed convex set (in the sense of inclusion), respectively, which contains $A$.
These sets are called the \emph{closure}, \emph{convex hull} and the \emph{closed convex hull}, respectively.
Given $h:X\to \overline{\R}$, where $\overline{\R}=[-\infty,+\infty]$\footnote{As is usual in convex analysis, we consider functions defined on the whole 
space; if it is not the case for some function $h$, we set $h(x)=+\infty$ for $x$ not in the domain of $h$}, we consider the following sets:
\begin{itemize}
 \item $\operatorname{dom}(h)=\{x\in X:~h(x)<+\infty\}$;
 \item $\operatorname{epi}(h)=\{(x,\lambda)\in X\times\R:~h(x)\leq\lambda\}$;
 \item for each $\lambda\in\R$, $S_{\lambda}(h)=\{x\in X:~h(x)\leq\lambda\}$.
\end{itemize}
The sets ${\rm dom}(h),~{\rm epi}(h)$ and $S_\lambda(h)$ are called the \emph{domain}, the \emph{epigraph} and the \emph{lower level set}
of $h$ with respect to $\lambda$, respectively.\\
\newline
Considering the convention $+\infty-\infty=-\infty+\infty=+\infty$, recall that a function $h:X\rightarrow\overline{\R}$ is said to be:
\begin{itemize}
\item \emph{convex} if, for all $x,y\in X$ and all $t\in[0,1]$,
$ h(x_t)\leq th(x)+(1-t)h(y)$,

\item \emph{quasiconvex} if, for all $x,y\in X$ and all $t\in\left[0,1\right]$,
$h(x_t)\leq\max\left\{h(x),h(y)\right\}$,

\item \emph{semistrictly quasiconvex}  if $h$ is quasiconvex and for all $x,y\in X$
\[
h(x)<h(y)\Rightarrow h(x_t)<h(y),~\forall t\in \left[0,1\right[,
\]
\end{itemize}
where $x_t=tx+(1-t)y$. It is clear that a convex function is quasiconvex and that the domain of a quasiconvex function is convex. 

We recall that $h$ is said to be \emph{lower semicontinuous} (in short lsc) at $x_0\in X$ if for all 
$\lambda<h(x_0)$, there exists a neighborhood $V$ of $x_0$ such that for all $x\in V$, it holds that $h(x)>\lambda$. Also,
$h$ is said to be lower semicontinuous if it is lower semicontinuous at any $x_0\in X$. 
A function $h$ is said to be \emph{upper semicontinuous} if $-h$ is lower semicontinuous. \\
\newline
Crouzeix defined in \cite{JPC} the regularizations of a function $h:X\to\overline{\R}$ as:
\begin{itemize}
\item $h_s(x)=\inf\{\lambda\in\R:~(x,\lambda)\in \overline{{\rm epi}(h)}\}$,
\item $h_c(x)=\inf\{\lambda\in\R:~(x,\lambda)\in{\rm co}({\rm epi}(h))\}$,
\item $h_{\overline{c}}(x)=\inf\{\lambda\in\R:~(x,\lambda)\in \overline{{\rm co}}({\rm epi}(h))\}$,
\item $h_q(x)=\inf\{\lambda\in\R:~x\in {\rm co}(S_{\lambda}(h))\}$ and
\item $h_{\overline{q}}(x)=\inf\{\lambda\in\R:~ x\in \overline{{\rm co}}(S_{\lambda}(h))\}$.
\end{itemize}
It results that $h_s, h_c, h_{\overline{c}}, h_q$ and $h_{\overline{q}}$ are the greatest lsc function (\emph{lsc regularization}), 
the greatest convex function (\emph{convex regularization}), the 
greatest lsc convex function ( \emph{lsc convex regularization}), 
the greatest quasiconvex function (\emph{quasiconvex regularization}) and 
the greatest lsc quasiconvex function (\emph{lsc quasiconvex regularization}) which are majorized by $h$, respectively.
It is clear that ${\rm epi}(h_s)=\overline{{\rm epi}(h)}$, ${\rm epi}(h_{\overline{c}})=\overline{{\rm co}}({\rm epi}(f))$,
${\rm epi}(h_{\overline{q}})=\overline{{\rm epi}(h_q)}$ and
\[
h_{\overline{c}} \leq h_{\overline{q}}\leq h_{q}\leq h.
\]
We say that a regularization $h_i$ of $h$ is \emph{well defined} when $h_i(x)\in\R$ for all $x\in{\rm dom}(h)$, where $i\in\{\overline{c},c,\overline{q},q,s\}$.\\
\newline
We recall some  different definitions of \emph{generalized monotonicity} (the ones we will be use from now on) for some bifunction $f:X\times X\rightarrow \R$:
\begin{itemize}
\item \emph{Quasimonotone} if, for all $x,y\in X$,
$f(x,y)>0\Rightarrow f(y,x)\leq0$.

\item \emph{Properly quasimonotone} if, for all $x_1,x_2,\dots, x_n\in X$, and all 
$x\in {\rm co}(\{ x_1,x_2,\dots,x_n\} )$, there exists $i\in \{ 1,2,\dots,n\}$ 
such that $f(x_i,x)\leq0$.
\item \emph{Pseudomonotone} if, for all $x,y\in X$,
$f(x,y)\geq0\Rightarrow f(y,x)\leq0$.
\item\emph{Monotone} if, for all $x,y\in X$, $f(x,y)+f(y,x)\leq0$.

\end{itemize}

Clearly, monotonicity implies pseudomonotonicity and this in turn implies quasimonotonicity. 
Nevertheless no relationship exists between quasimonotonicity and proper quasimonotonicity
of bifunctions (e.g. \cite{BP01}). On the other hand, all the bifunctions $f$ satisfying 
some property of generalized monotonicity mentioned above satisfy $ f(x,x)\leq 0$ \ for all $x\in X$.

Let $K$ be a convex subset of $X$. A bifunction $f:K\times K\to\R$ is said to have  the 
\begin{itemize}
\item \emph{local upper sign property} at $x\in K$ if there exists $r>0$ such that for every $y\in K\cap B(x,r)$ 
the following implication holds:
\begin{eqnarray}\label{upperlocal}
\bigl(
f(x_t,x)\leq0,~
\forall~ t\in\,]0,1[~
\bigr)
\Rightarrow ~ f(x,y) \geq0,
\end{eqnarray}

\item \emph{upper sign property} at $x\in K$ if for every $y\in K$ 
the following implication holds:
\begin{eqnarray}\label{upper}
\bigl(
f(x_t,x)\leq0,~
\forall~ t\in\,]0,1[~
\bigr)
\Rightarrow ~ f(x,y) \geq0,
\end{eqnarray}

\end{itemize}
where $x_t=(1-t)x+ty$.\\
\newline
For example, any positive bifunction has the upper sign property.
Addionally, any bifunction such that 
$f(x,x)\geq0$, $f(\cdot,y)$ is upper semicontinuous and $f(x,\cdot)$ is semistrictly quasiconvex, for all $x,y\in K$, has the upper sign property.
Clearly, every bifunction with the upper sign property has the local upper sign property.
Moreover, in \cite{ACI}, Aussel \emph{et al.} showed that these concepts are equivalent under the following condition:
\begin{equation}\label{beta}
f(x,y)<0 \mbox{ and }f(x,x)=0~\Rightarrow~f(x,x_t)<0~\forall t\in]0,1[,
\end{equation}
where $x_t=tx+(1-t)y$. In particular, this holds when $f(x,\cdot)$ is a semistrictly quasiconvex bifunction.

\section{Regularization of a bifunction}
From now on, $X$ stands for a real Banach space and $f:K\times K\to \R$ for a bifunction defined on a nonempty and closed convex subset $K$ of
$X$. For each $x\in K$,
we denote by $f_s(x,\cdot),~f_c(x,\cdot),~f_{\overline{c}}(x,\cdot),~f_q(x,\cdot)$  and
$f_{\overline{q}}(x,\cdot)$ the lower semicontinuous, convex, convex and lower semicontinuous, quasiconvex and
quasiconvex and lower semicontinuous reguralizations of the function $f(x,\cdot)$, respectively.

Clearly, for every $x,y\in K$ holds that:
\begin{eqnarray}\label{desi1}
f_c(x,y)\leq f_q(x,y)\leq f(x,y)
~\mbox{ and }~
f_{\overline{c}}(x,y)\leq  f_{\overline{q}}(x,y)\leq f_s(x,y)\leq f(x,y).
\end{eqnarray}
In general, $f_i(x,y)$ can be $-\infty$, where $i\in\{s,c,q,\overline{c},\overline{q}\}$.\\
\newline
We define the following families of bifunctions depending on $K$:
\begin{itemize}
 \item {$\mathcal{C}(K)=\{f:K\times K\to\R:~f_{c}(x,\cdot)\mbox{ is well defined for all }x\in K\}$.}
 \item {$\mathcal{Q}(K)=\{f:K\times K\to\R:~f_{q}(x,\cdot)\mbox{ is well defined for all }x\in K\}$.}
 \item ${\overline{\mathcal{C}}(K)}=\{f:K\times K\to\R:~f_{\overline{c}}(x,\cdot)\mbox{ is well defined for all }x\in K\}$.
 \item ${\overline{\mathcal{Q}}(K)}=\{f:K\times K\to\R:~f_{\overline{q}}(x,\cdot)\mbox{ is well defined for all }x\in K\}$.
 \item $\mathcal{S}(K)=\{f:K\times K\to\R:~f_s(x,\cdot)\mbox{ is well defined for all }x\in K\}$. 
\end{itemize}
It is clear from (\ref{desi1}) that:
\begin{equation}\label{inclusions-K}
{\mathcal{C}(K)\subset \mathcal{Q}(K)~\mbox{ and }}~\overline{\mathcal{C}}(K)\subset \overline{\mathcal{Q}}(K)\subset \mathcal{S}(K).
\end{equation}
The following example shows that the previous inclusions are strict in general.
\begin{example}
Let $K=\R$ and let $f_1,f_2:\R\times\R\to\R$ two bifunctions defined as
$f_1(x,y)=y^3-x$ for all $(x,y)\in\R^2$ and 
\[
 f_2(x,y)=\left\lbrace\begin{array}{cc}
                     -\ln(|y|),& y\neq0,\\
                     0,&y=0.
                    \end{array}
\right. 
\]
For each $x\in\R$ we have the following graphs:
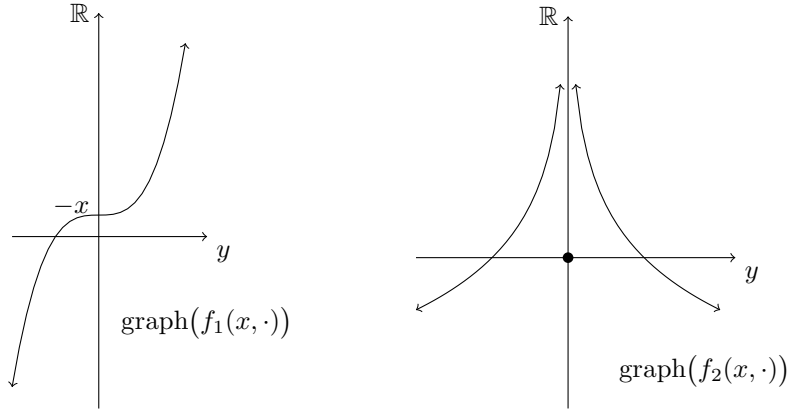
\begin{center}
\begin{tikzpicture}[scale=0.57]
\draw[->](-2,0)--(2.5,0)node[below right]{$y$}; 
\draw[->](0,-4)--(0,5.2)node[left]{$\R$};
\draw[domain=-2:2,<->]plot(\x,{0.5*(\x)^3+0.5});
\draw(2.5,-2)node[]{${\rm graph}\big(f_1(x,\cdot)\big)$};
\draw(0,0.7)node[left]{$-x$};
\end{tikzpicture}
\quad\quad\quad\quad
\begin{tikzpicture}[scale=1]
\draw[->](-2,0)--(2.2,0)node[below right]{$y$}; 
\draw[->](0,-2)--(0,3.2)node[left]{$\R$};
\draw[domain=-2:-0.1,<->]plot(\x,{-ln(-\x)});
\draw[domain=0.1:2,<->]plot(\x,{-ln(\x)});
\fill(0,0)circle(2pt);
\draw(1.8,-1.5)node[]{${\rm graph}\big( f_2(x,\cdot)\big)$};
\end{tikzpicture}
\end{center}
Clearly, $f_1\in\mathcal{Q}(K)\setminus\mathcal{C}(K)$ and $f_2\in \mathcal{S}(K)\setminus \overline{\mathcal{Q}}(K)$.
\end{example}
\noindent The following result shows that under compactness of $K$ the three families are the same.
\begin{proposition}
Let $K\subset X$ be a nonempty and compact convex set. Then 
\[
\overline{\mathcal{C}}(K)=\overline{\mathcal{Q}}(K)=\mathcal{S}(K).
\]
\end{proposition}
\begin{proof}
In view of the inclusions (\ref{inclusions-K}), it is enough to show that $\mathcal{S}(K)\subset \overline{\mathcal{C}}(K)$.
Let $f\in\mathcal{S}(K)$ and let $x\in K$. In view of the compactness of $K$ and the lower semicontinuity of $f_s(x,\cdot)$,
there exists $x_0\in K$ such that $f(x,y)\geq f_{s}(x,x_0)\mbox{ for all }y\in K$ and
consequently 
\[
{\rm epi}(f(x,\cdot))\subset K\times[f_{s}(x,x_0),+\infty[.
\]
It follows that ${\rm epi}(f_{\overline{c}}(x,\cdot))\subset K\times[f_{s}(x,x_0),+\infty[$, concluding that $f\in\overline{\mathcal{C}}(K)$.
\end{proof}
Castellani \emph{et al.} \cite{castellani2010}
considered the family of bifunctions $f$ such that $f(x,x)=0$ for all $x\in K$ and
satisfying the following condition:
\begin{eqnarray}\label{fam-convex}
\forall x\in K,~\exists x^*\in X^*,~\exists a\in\R: \forall y\in K,~\langle x^*,y\rangle+a\leq f(x,y),
\end{eqnarray}
in a finite dimensional space. The following result shows that the family ${\overline{\mathcal{C}}(K)}$ 
is also characterized by the condition (\ref{fam-convex}) in an infinite dimensional space.
\begin{proposition}
The family ${\overline{\mathcal{C}}(K)}$ is the set of bifunctions $f$ satisfying the condition (\ref{fam-convex}).
\end{proposition}
\begin{proof}
Let $f\in{\overline{\mathcal{C}}(K)}$ and let $x\in K$.
Without loss of generality we can assume that $f(x,\cdot)$ is 
a convex and lower semicontinuous function, for all $x\in K$.
From \cite[Theorem I.7]{Brezis} we have that for each $(x_0,\lambda)\in K\times\R\setminus {\rm epi}(f(x,\cdot))$
there exists $(x_0^*,\lambda^*)\in X^*\times \R$  such that 
\begin{equation}\label{separation}
\langle x_0^*,x_0\rangle+\lambda^*\lambda< \langle x_0^*,y\rangle+ \lambda^*f(x,y)\mbox{ for all }y\in K. 
\end{equation}
By substituting $y=x_0$ into (\ref{separation}) we obtain $\lambda^*>0$. Thus,
\[
f(x,y)\geq \left\langle -\frac{x_0^*}{\lambda^*},y\right\rangle +\lambda+\left\langle \frac{x_0^*}{\lambda^*},x_0\right\rangle
\mbox{ for all }y\in K. 
\]
Therefore, the bifunction $f$ satisfies the condition (\ref{fam-convex}) with $x^*=-\displaystyle\frac{x_0^*}{\lambda^*}$
and $a=\lambda+\displaystyle\left\langle\frac{x_0^*}{\lambda^*},x_0\right\rangle$.\\
\newline
Conversely, let $f:K\times K\to\R$ be a bifunction satisfying (\ref{fam-convex}).
In view of the convexity and lower semicontinuity of
the function $h_x:K\to\R$, defined as $h_x(y)=\langle x^*,y\rangle+a$,
the bifunction $f\in{\overline{\mathcal{C}}(K)}$.
\end{proof}


It is natural to ask whether some kind generalized {monotonicity} of a bifunction is shared with its {regularizations}. 
The following {lemma} is a key step towards this result.

\begin{lemma}\label{prop-mono-upper}
Let  $f,g: K\times K\to\R$ be two bifunctions
such that 
\begin{align} \label{gleqf}
 g(x,y)\leq f(x,y) \mbox{ for all } x,y\in K.
\end{align}
If $f$ is either {monotone}, pseudomonotone, quasimonotone or properly quasimonotone bifunction, then $g$ is a bifunction of the same type of monotonicity.
\end{lemma}

\begin{proof}
In the case $f$ is {monotone} [respectively {pseudomonotone or} quasimonotone], the inequalities
\begin{equation*}
  g(x,y)\leq f(x,y)\mbox{ and } g(y,x)\leq f(y,x)\mbox{ for all } x,y\in K
\end{equation*}
imply the {motonicity} [respectively {pseudomonotonicity or} quasimonotonicity] of $g$. 

Now, assume that $f$ is properly quasimonotone. Let  $x_1,x_2,\dots,x_m\in K$ and let $x\in{\rm co}(x_1,x_2,\dots,x_m)$. Then, there exists $j_0\in\{1,2,\dots,m\}$
such that 
\[
f(x_{j_0},x)=\min_{j\in\{1,2,\dots,m\}} f(x_j,x) \leq 0,
\]
and consequently {by \eqref{gleqf}}
$$\displaystyle\min_{j\in\{1,2,\dots,m\}} g(x_j,x)\leq g(x_{j_0},x)\leq f(x_{j_0,x})\leq0,$$
which shows that $g$ is a properly quasimonotone bifunction.
\end{proof}

Now, as a direct consequence of inequalities (\ref{desi1}) and Proposition \ref{prop-mono-upper} we have the following corollary.

\begin{theorem}\label{preserva-mono}
If a bifunction is either monotone, {pseudomonotone}, quasimonotone or properly quasimonotone, then all its regularizations have the same type of generalized monotonicity.
\end{theorem}

\begin{remark}
The converse of the last result is not true in general. We consider for instance the bifunction
$f:K\times K\to\R$ defined as:
\[
 f(x,y)=\left\lbrace\begin{matrix}
                     1,&(x,y)=(1,0)~\vee~(x,y)=(0,1)\\
                     0,& (x,y)\neq(1,0)~\wedge~(x,y)\neq(0,1)
                    \end{matrix}\right.
\]
where $K=[0,1]$. The following pictures represent the graphs of the functions  $f(x,\cdot)$:
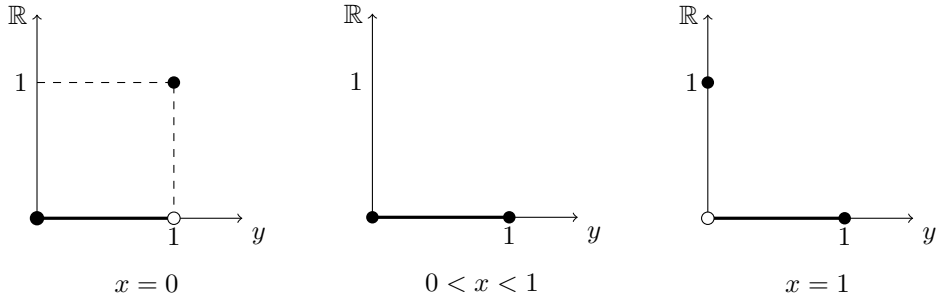
\begin{center}
\begin{tikzpicture}[scale=1.8]
\draw[->](0,0)--(1.5,0)node[below right]{$y$}; 
\draw[->](0,0)--(0,1.5)node[left]{$\R$};
\draw[very thick](0,0)--(1,0);
\fill[white](1,0)circle(1.3pt);
\draw(1,0)circle(1.3pt);
\fill(1,1)circle(1.3pt);
\draw[dashed](1,0.1)--(1,1);
\draw[dashed](0,1)--(1,1);
\draw(1,0)node[below]{$1$};
\draw(0,1)node[left]{$1$};
\draw(0.8,-0.35)node[below]{$x=0$};
\fill(0,0)circle(1.5pt);
\end{tikzpicture}
\quad\quad
\begin{tikzpicture}[scale=1.8]
\draw[->](0,0)--(1.5,0)node[below right]{$y$}; 
\draw[->](0,0)--(0,1.5)node[left]{$\R$};
\draw[very thick](0,0)--(1,0);
\fill(1,0)circle(1.3pt);
\fill(0,0)circle(1.3pt);
\draw(1,0)node[below]{$1$};
\draw(0,1)node[left]{$1$};
\draw(0.8,-0.35)node[below]{$0<x<1$};
\end{tikzpicture}
\quad\quad
\begin{tikzpicture}[scale=1.8]
\draw[->](0,0)--(1.5,0)node[below right]{$y$}; 
\draw[->](0,0)--(0,1.5)node[left]{$\R$};
\draw(0.8,-0.35)node[below]{$x=1$};
\draw[very thick](0,0)--(1,0);
\fill[white](0,0)circle(1.3pt);
\draw(0,0)circle(1.3pt);
\fill(0,1)circle(1.3pt);
\fill(1,0)circle(1.3pt);
\draw(1,0)node[below]{$1$};
\draw(0,1)node[left]{$1$};
\end{tikzpicture}
\end{center}
The bifunction $f$ is not monotone, because $f(1,0)+f(0,1)=2 $. 
On the other hand, for all $ x,y\in K$ and for all $i\in\{\overline{c},\overline{q},s\}$
we have $f_i(x,y)=0$. Therefore $f_i$ is mononote for all $i\in\{\overline{c},\overline{q},s\}$.
\end{remark}

The following example shows that some bifunction can have the upper sign property without none of its regularizations having it.

\begin{example}\label{exam-upper-quasi}
Let $K=[0,1]$ and let $f:K\times K\to\R$ be a bifunction defined as 
\[
 f(x,y)=\left\lbrace\begin{array}{cl}
                     0,& x,y\in K~\wedge~[\{x,y\}\subset\Q~\vee~\{x,y\}\cap\Q=\emptyset]\\
                     1,&y\in K\setminus\Q~\wedge~x\in\Q\cap K\\
                     -1,&x\in K\setminus\Q~\wedge~y\in\Q\cap K,
                    \end{array}\right.
\]
where $\Q$ is the set of rational numbers. Let $x,y\in K$ such that
\begin{equation}\label{upper-sign-example1}
f(x_t,x)\leq0   \mbox{ for all } t \in [0,1].
\end{equation}
If $x\in\Q$, then $f(x,y)\geq0$.
However, if $x\notin \Q$, then
\[
f(x_t,x)=\left\lbrace\begin{array}{cl}
                      0,&x_t\notin \Q,\\
                      1,&x_t\in\Q.
                     \end{array}
 \right.
\]
From (\ref{upper-sign-example1}) we have $x_t\notin\Q$, for all $t\in]0,1[$.
It follows that $x=y$ and consequently $f(x,y)=0$. Therefore, 
the bifunction $f$ has the upper sign property.

On the other hand, for all $i\in\{\overline{c},\overline{q},s\}$ it holds that
\[
 f_i(x,y)=\left\lbrace
\begin{matrix}
0,&x\in \Q\cap K\\
-1,&x\in K\setminus \Q
\end{matrix}
\right.
\]
and the regularization $f_i$ does not have the upper sign property on $K$, because taking $x=\sqrt{2}/2$ and $y=0$, we have
$f(x,y)=-1$ and $f(t\sqrt{2}/2,0)\leq0$, for all $t\in]0,1[$.
\end{example}

In contrast to our result on generalized {monotonicity}, where from the inequality \eqref{gleqf} 
the property is transmitted from the bifunction $f$ to $g$, the upper sign property is transmitted from $g$ to $f$.

\begin{lemma}\label{prop-upper}
Let  $f,g: K\times K\to\R$ be two bifunctions satisfying \eqref{gleqf}. 
If $g$ has the upper sign property on $K$, then $f$ also has this property.
\end{lemma}

\begin{proof}
Let $x,y \in K$ such that $f(x_t,x)\leq 0$ for all $t\in]0,1[$. 
Since $g(x,y)\leq f(x,y)$ for all $x,y\in K$, then
$g(x_t,x)\leq 0$ for all $t\in]0,1[$.
The upper sign property of $g$ implies that
$g(x,y)\geq0$, and therefore that $f(x,y)\geq 0$. 
\end{proof}
\begin{remark}
With the hypothesis of Proposition \ref{prop-upper}, if $g$ has local upper sign property, then $f$ also has the local upper sign property.
\end{remark}

From Proposition \ref{prop-upper} and the inequalities on (\ref{desi1}) we have the following result.
\begin{theorem}\label{preserva-upper}
If some regularization of a bifunction has the (local) upper sign property,
then the bifunction also has the (local) upper sign property. 
\end{theorem}
The following result states that the quasiconvex regularization of a bifunction is upper semicontinuous on the second variable, provided
that the bifunction also is.

\begin{proposition}\label{fq-upper}
{Let $f\in\mathcal{Q}(K)$. 
If  $f(\cdot, y)$ is an upper semicontinuous function for every $y\in K$,
then  $f_q$ is also upper semicontinuous with respect to first argument.}
\end{proposition}
\begin{proof}
For every $\varepsilon>0$ and $\lambda\in\R$ such that $f_q(x,y)<\lambda<f_q(x,y)+\varepsilon$.
Then  $y\in{\rm co}(S_{\lambda}(f(x,\cdot)))$ and this implies that
there exists $y_1,\dots,y_m\in S_{\lambda}(f(x,\cdot))$ and $t_1,\dots,t_m\in[0,1]$ such that
$\sum_{k=1 }^m t_k=1$ and $y=\sum_{k=1}^mt_ky_k$.
Since $f(\cdot,y_k)$ is upper semicontinuous, there exists
a neigborhood $V$ of $x$ such that for all $x'\in V$ and all $k\in\{1,\dots,m\}$ hold
$f(x',y_k)<f_q(x,y)+\varepsilon$.
Therefore, 
$
f_q(x',y_k)\leq f(x',y_k)<f_q(x,y)+\varepsilon$ for each $k\in\{1,\dots,m\}$ and for all $x'\in V$.
So, by the quasiconvexity of $f_q(x,\cdot)$, we have that
$ f_q(x',y)<f_q(x,y)+\varepsilon$ for all $x'\in V$.
\end{proof}

\begin{remark}\label{Nash-important}
{The previous result is also true for the convex 
regularizations.}
\end{remark}
We define the following subfamily of $\mathcal{Q}(K)$:
\[
 \mathcal{SQ}(K)=\{f\in \mathcal{Q}(K):~ {f_{q}}(x,\cdot)\mbox{ is semistrictly quasiconvex for all }x\in K\}
\]
Clearly, $\mathcal{C}(K)\subset  \mathcal{SQ}(K)\subset  \mathcal{Q}(K)$.

%

As we have mentioned earlier, the local sign property and the sign property are equivalent under 
condition \eqref{beta}, a condition that holds for ${f_{q}}$ with $f \in \mathcal{SQ}(K)$, but no for $f$ itself.

%

\begin{example}
Let $K=[0,2]$ and let $f:K\times K\to\R$ be a bifunction defined by:
\[
f(x,y)=\left\lbrace\begin{array}{cc}
 y-2,&y\neq 1\\
 0,&y=1
\end{array}
\right. 
\]
It is not difficult to see that ${ f_{q}}(x,y)=y-2$ for all $x,y\in K$. Therefore, $f\in \mathcal{SQ}(K)$.
However $f$ does  not satisfy the condition (\ref{beta}). Indeed, taking $y_1=0$ and $y_2=2$ we have $f(x,y_1)<0$ and $f(x,y_2)=0$, but $f(x,1)=0$.
\end{example}

\begin{proposition}
Let $f\in\mathcal{SQ}(K)$ be a bifunction such that $f$ is upper semicontinuous  with respect to first variable.
Then $f_q$ has the upper sign property if, and only if, $f_q(x,x)\geq0$, for all $x\in K$.
\end{proposition}
\begin{proof}
It is clear that if  $f_q$ has the upper sign property then $f_q(x,x)\geq0$, for all $x\in K$.

The converse, let $x,y\in K$ such that $f_q(x_t,x)\leq0$, for every $x_t=tx+(1-t)y$, $t\in]0,1[$.
If there exists $t\in]0,1[$ such that $f_q(x_t,y)<0$ then by semistrictly quasiconvexity of $f_q(x_t,\cdot)$ we have 
$f_q(x_t,x_t)<0$, which is a contradiction. So, $f_q(x_t,y)\geq0$ for all $t\in]0,1[$. 
By Proposition \ref{fq-upper} $f_q$
is upper semicontinuous with respect to first variable and this imply $f(x,y)\geq0$.
\end{proof}

\section{Equilibrium Problems vs Convex Feasibility Problems}
We denote by $\operatorname{EP}(f,K)$ and $\operatorname{CFP}(f,K)$ the solution sets of \eqref{EP} and \eqref{cfp}, respectively.

\begin{lemma}\label{g-f}
Let $f,g:K\times K\to\R$ be two bifunctions {satisfying \eqref{gleqf}.}
Then $\operatorname{EP}(g,K)\subset \operatorname{EP}(f,K)$ and $\operatorname{CFP}(f,K)\subset \operatorname{CFP}(g,K)$.
\end{lemma}
\begin{proof}
It follows directly from definitions of \eqref{EP} and \eqref{cfp}. 
\end{proof}

Clearly, \eqref{desi1} and Lemma \ref{g-f} imply that if $f\in{\overline{\mathcal{C}}(K)}$ then
\begin{equation}\label{inclusion-EP}
\operatorname{EP}(f_{\overline{c}},K)\subset \operatorname{EP}(f_{\overline{q}},K)\subset \operatorname{EP}(f_s,K)\subset \operatorname{EP}(f,K)
\end{equation}
and
\begin{equation}\label{inclusion-CFP}
\operatorname{CFP}(f,K)\subset \operatorname{CFP}( f_s,K)\subset \operatorname{CFP}(f_{\overline{q}},K) \subset \operatorname{CFP}(f_{\overline{c}},K).
\end{equation}

The following result says that if a bifunction admits convex {and lower semicontinuous} regularization, then the solution sets
of equilibrium problem associated this bifunction and its regularizations are the same.
\begin{proposition}\label{igualdad}
If $f \in {\overline{\mathcal{C}}(K)}$ then 
$\operatorname{EP}(f_i,K)=\operatorname{EP}(f,K)$ for all $i\in\{c,q,s,\overline{c},\overline{q}\}$.
\end{proposition}
\begin{proof}
By (\ref{inclusion-EP}) it is enough to show that  $\operatorname{EP}(f,K)\subset \operatorname{EP}(f_{\overline{c}},K)$.
For each
$x\in \operatorname{EP}(f,K)$ we have that $\operatorname{epi}(f(x,\cdot))\subset K\times[0,+\infty[$ and consequently
$\operatorname{epi}(f_{\overline{c}}(x,\cdot))\subset K\times[0,+\infty[$, i.e. $x\in \operatorname{EP}(f_{\overline{c}},K)$.
\end{proof}
\begin{remark}\label{regu-semi-equivalente}
In Proposition \ref{igualdad}:
\begin{itemize}
 \item If $f\in{\overline{\mathcal{Q}}(K)}$ then $\operatorname{EP}(f_i,K)=\operatorname{EP}(f,K)$, for all $i\in\{q,s,\overline{q}\}$.
  \item If $f\in\mathcal{S}(K)$ then $\operatorname{EP}(f_s,K)=\operatorname{EP}(f,K)$.
\item {If $f\in\mathcal{C}(K)$ then $\operatorname{EP}(f_i,K)=\operatorname{EP}(f,K)$, for all $i\in\{c,q\}$.}
\item {If $f\in\mathcal{Q}(K)$ then $\operatorname{EP}(f_q,K)=\operatorname{EP}(f,K)$.}
\end{itemize}
\end{remark}
The inclusions in (\ref{inclusion-CFP}) and Proposition \ref{igualdad} motivate the following question: 
Do a bifunction and its regularizations have the same solution set for the convex feasibility problem?
The following example gives a negative answer.
\begin{example}
Let $K=[0,1]$ and let $f:K\times K\to\R$ defined as 
\[
 f(x,y)=\left\lbrace\begin{array}{cl}
                     y,&y\in[0,1[\\
                     0,&y=1
                    \end{array}
\right.
\]
For every $x\in K$, the graph of $f(x,\cdot)$ is
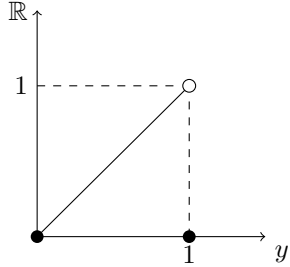
\begin{center}
\begin{tikzpicture}[scale=2]
\draw[->](0,0)--(1.5,0)node[below right]{$y$}; 
\draw[->](0,0)--(0,1.5)node[left]{$\R$};
\draw(0,0)--(1,1);
\fill[white](1,1)circle(1.2pt);
\draw(1,1)circle(1.2pt);
\fill(0,0)circle(1.2pt);
\fill(1,0)circle(1.2pt);
\draw[dashed](1,0)--(1,0.9);
\draw(1,0)node[below]{$1$};
\draw(0,1)node[left]{$1$};
\draw[dashed](0,1)--(0.9,1);
\end{tikzpicture}
\end{center}
It is not difficult to see that $\operatorname{CFP}(f,K)=\{0,1\}$. On the other hand,
$f_{\overline{c}}(x,y)=0$ for all $x,y\in K$ and this implies $\operatorname{CFP}(f_{\overline{c}},K)=[0,1]$.
\end{example}

Bianchi and Pini \cite{BP05} considered a weaker concept of solution for \eqref{cfp}, similar that the one 
proposed by Aussel and Hadjisavvas \cite{AH04} in the setting of variational inequalities. They define the set of the \emph{local solutions} 
\[
\operatorname{CFP}_{\operatorname{local}} (f,K)=\{x\in K:~\exists r>0\mbox{ s.t. }f(y,x)\leq0~\forall y\in K\cap B(x,r)\}.
\]
Clearly, $\operatorname{CFP}(f,K)\subset \operatorname{CFP}_{\operatorname{local}} (f,K)$.

In the following result, part \emph{(i)}  is from \cite[Theorem 1]{castellani2012}, 
and part \emph{(ii)} is an adaptation of \cite[Proposition 3.1]{ACI}. 

\begin{proposition}\label{caste-ACI} Let $f:K\times K\to\R$ be a bifunction.
\begin{enumerate}
 \item[(i)] \label{ffkk} If $f$ has the local upper sign property and satisfies (\ref{beta}) then 
 $\operatorname{CFP}_{\operatorname{local}}(f,K)\subset \operatorname{EP}(f,K)$.
 \item[(ii)] If $f$ has the upper sign property then $\operatorname{CFP}(f,K)\subset \operatorname{EP}(f,K)$.
\end{enumerate}
\end{proposition}

The last result shows that, in order to obtain a solution of the equilibrium problem, it is enough to obtain a solution for the convex feasibility problem under the upper sign property, or under the local upper sign property and (\ref{beta}).
As a consequence of Proposition \ref{caste-ACI}, (\ref{inclusion-CFP}) and Remark \ref{regu-semi-equivalente}, we have the following result.

\begin{proposition}\label{extension-ACI}
Let $f\in \mathcal{S}(K)$.
\begin{enumerate}
 \item[(i)] If $f_s$ has the local upper sign property and satisfies the condition (\ref{beta}) then 
 $ \operatorname{CFP}_{\operatorname{local}}(f_s,K)\subset \operatorname{EP}(f,K)$.

 \item[(ii)] If $f_s$ has the upper sign property then $\operatorname{CFP}(f_s,K)\subset \operatorname{EP}(f,K)$.
 \item[(iii)] If $f\in \mathcal{SQ}(K)$ and 
 ${f_{q}}$ has the upper sign property then 
 $\operatorname{CFP}_{\operatorname{local}}({f_{q}},K)\subset \operatorname{EP}(f,K)$.
\end{enumerate}
\end{proposition}

The following examples show that the nonemptiness of the solution set of an equilibrium problem cannot be directly deduced  from Proposition \ref{caste-ACI}.

\begin{example}\label{dos-eje}
Let $K=\R$ and let $f:K\times K\to\R$ be a bifunction defined by:
\[
f(x,y)=\left\lbrace\begin{array}{cl}
                    0,& x,y\in \Q\cap K~\vee~ x=y\\
                    1,& \mbox{otherwise}.
                   \end{array}\right. 
\]
The bifunction $f$ has the upper sign property on $K$, and
it is not difficult to show that $\operatorname{CFP}_{\operatorname{local}}(f,K)=\emptyset$.

Moreover, $f_s(x,y)=0$ for all $x, y \in K$, which implies that $f_s$ is properly quasimonotone, it has the upper sign property on $K$ and 
$\operatorname{CFP}_{\operatorname{local}}(f_s,K)=K$.
Therefore, by Proposition \ref{extension-ACI} \emph{(i)} we have $\operatorname{EP}(f,K)$ is nonempty.


\end{example}

\begin{example}
Let $K=[0,+\infty[$ and let $f:K\times K\to\R$ defined by:
\[
f(x,y)=\left\lbrace\begin{array}{cl}
                    0,& y=0\\
                    1/y,& y\neq0.
                   \end{array}\right. 
\]
\end{example}
It is not difficult to see that 
$\operatorname{CFP}_{\operatorname{local}}(f_s,K)=\operatorname{CFP}_{\operatorname{local}}(f,K)=\emptyset$. 
On the other hand, ${f_{q}}(x,y)=0$
for all $x,y\in K$. Thus, $f\in \mathcal{SQ}(K)$ and ${f_{q}}$ has the upper sign property.
Moreover, $\operatorname{CFP}_{\operatorname{local}}({f_{q}},K)=K$ and by 
Proposition \ref{extension-ACI} \emph{(iii)} we have that $\operatorname{EP}(f,K)$ is nonempty. 

\section{Existence results}


In 1972, Ky Fan proved his famous minimax inequality (cf. \cite[Theorem 1]{Kfan}).

\begin{theorem}[Ky Fan, 1972] \label{Ky1}
Let $V$ be a real Hausdorff topological vector space and $K$ a nonempty compact convex subset of $V$. If
$f:K\times K\to\R$ satisfies:
\begin{itemize}
 \item[$(i)$] $f(\cdot, y):K\to\R$ is upper semicontinuous for each $y\in K$,
 \item[$(ii)$] $f(x,\cdot): K\to\R$ is quasiconvex for each $x\in K$,
\end{itemize}
then there exists a point $x^*\in K$ such that
\[
 \inf_{y\in K}f(x^*,y)\geq \inf_{w\in K}f(w,w).
\]
\end{theorem}

The following result is a consequence of Theorem \ref{Ky1} and Proposition \ref{igualdad}.

\begin{coro}\label{R1}
Let  $K$ be a nonempty compact convex subset of  $X$ and $f\in{\mathcal{Q}(K)}$.
If the following assumptions hold:
\begin{enumerate}
 \item[(i)] $f_q(\cdot,y)$ is upper semicontinuous for all $y\in K$,
 \item[(ii)] $f_q(x,x)=0$ for all $x\in K$.
\end{enumerate}
Then $\operatorname{EP}(f,K)$ is a nonempty set.
\end{coro}

\begin{remark}
If $f(\cdot,y)$ is upper semicontinuous for all $y\in K$, by Proposition \ref{fq-upper}
the assumption \emph{(i)} in  Corollary \ref{R1} is satisfied.
\end{remark}
\begin{example}
Let $K=[0,2]$ and let $f:K\times K\to\R$ be a bifunction defined by
\[
f(x,y)=\left\lbrace\begin{array}{cl}
                    0,& y\in[0,x[\\
                    -y+x,& y\in[x,2] ~\wedge~x\neq1\\
                    y-1,&y\in]1,2]~\wedge~x=1\\
                    1,& x=y=1.
                   \end{array}\right.
\]
The following pictures represent the graphs of the functions $f(x,\cdot)$:
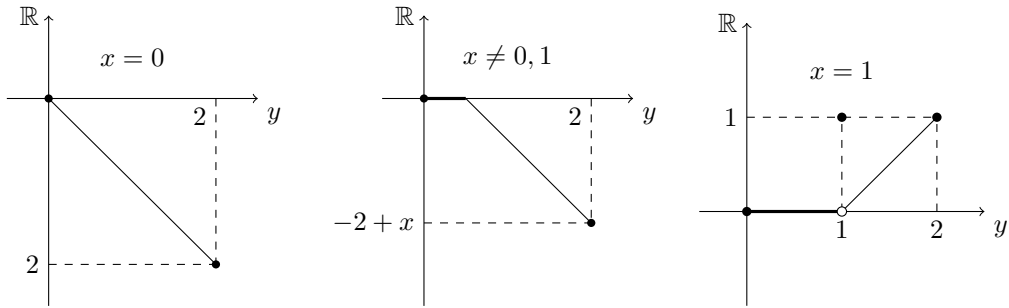
\begin{center}
\begin{tikzpicture}[scale=1.1]
\draw[->](-0.5,0)--(2.5,0)node[below right]{$y$}; 
\draw[->](0,-2.5)--(0,1)node[left]{$\R$}; 
\draw(0,0)--(2,-2);
\fill(0,0)circle(1.4pt);
\fill(2,-2)circle(1.4pt);
\draw[dashed](2,0)--(2,-2);
\draw[dashed](0,-2)--(2,-2);
\draw(2,0)node[below left]{$2$};
\draw(0,-2)node[left]{$2$};
\draw(1,0.5)node[]{$x=0$};
\end{tikzpicture}
\quad
\begin{tikzpicture}[scale=1.1]
\draw[->](-0.5,0)--(2.5,0)node[below right]{$y$}; 
\draw[->](0,-2.5)--(0,1)node[left]{$\R$}; 
\draw[very thick](0,0)--(0.5,0);
\draw(0.5,0)--(2,-1.5);
\fill(0,0)circle(1.4pt);
\fill(2,-1.5)circle(1.4pt);
\draw[dashed](2,0)--(2,-1.5);
\draw[dashed](0,-1.5)--(2,-1.5);
\draw(2,0)node[below left]{$2$};
\draw(0,-1.5)node[left]{$-2+x$};
\draw(1,0.5)node[]{$x\neq 0, 1$};
\end{tikzpicture}
\quad
\begin{tikzpicture}[scale=1.25]
\draw[->](-0.5,0)--(2.5,0)node[below right]{$y$}; 
\draw[->](0,-1)--(0,2)node[left]{$\R$}; 
\draw[very thick](0,0)--(1,0);
\draw[domain=1:2]plot(\x,{\x-1});
\draw[dashed](1,0)--(1,1);
\draw[dashed](0,1)--(2,1);
\fill[white](1,0)circle(1.4pt);
\fill(0,0)circle(1.4pt);
\fill(2,1)circle(1.4pt);
\draw(1,0)circle(1.4pt);
\fill(1,1)circle(1.4pt);
\draw(1,0)node[below]{$1$};
\draw(0,1)node[left]{$1$};
\draw[dashed](2,0)--(2,1);
\draw(2,0)node[below]{$2$};
\draw(1,1.5)node[]{$x=1$};
\end{tikzpicture}
\end{center}
Clearly, $f(x,\cdot)$ is quasiconvex and continuous for all $x\neq 1$, but $f(1,\cdot)$ is not
quasiconvex. Nevertheless, $f_q(1,\cdot)$ is quasiconvex and continuous.
Therefore, $f\neq f_q$ and moreover $f_q(x,x)=0$ for all $x\in K$. Applying Corollary \ref{R1}
$\operatorname{EP}(f,K)$ is nonempty. Notice that the nonemptiness of $\operatorname{EP}(f,K)$ cannot be directly deduced from Theorem \ref{Ky1}.
\end{example}




The following result is a consequence of \cite[Proposition 2.1]{ACI}. 


\begin{proposition} \label{ACI1}
Let $K$ be a weakly compact subset of $X$ and $f$ be a properly quasimonotone bifunction such that for every $x\in K$
the set $\{y\in K:~f(x,y)\leq0\}$ is weakly closed. Then $\operatorname{CFP}(f,K)$  is nonempty.
\end{proposition}

Since every quasiconvex and lower semicontinuous function is lower semicontinuous in the weakly topology, the application of Proposition \ref{ACI1} gives us the following result.

\begin{cor}\label{CFP}
Let $K$ a weakly compact subset of $X$ and $f\in {\overline{\mathcal{Q}}(K)}$. If $f_{\overline{q}}$ is properly quasimonotone then $\operatorname{CFP}(f_{\overline{q}},K)$
is nonempty. Moreover, if $f_{\overline{q}}$ has the upper sign property then $\operatorname{EP}(f,K)$ is nonempty.
\end{cor}

\begin{proof}
Clearly, by Proposition \ref{ACI1}, $\operatorname{CFP}(f_{\overline{q}},K)$ is a nonempty set. Since $f_{\overline{q}}$ has the upper sign property, 
Proposition \ref{caste-ACI} \emph{(ii)} implies that $\operatorname{EP}(f_{\overline{q}},K)\neq\emptyset$. The result follows from 
{Remark  \ref{regu-semi-equivalente}}. 
\end{proof}

Another consequence of Proposition \ref{ACI1} is the following one.

\begin{cor}\label{R2}
Let $K$ be a nonempty compact convex subset of  $X$ and $f\in\mathcal{S}(K)$ such that
$f_s$ is properly quasimonotone and it has the upper sign property.
Then $\operatorname{EP}(f,K)$ is nonempty.
\end{cor}

\begin{proof}
By Proposition \ref{ACI1} we have that $\operatorname{CFP}(f_s,K)$ is nonempty. The result follows from Proposition \ref{extension-ACI} \emph{(ii)}. 
\end{proof}
For each $n\in\N$, let $K_n=\{x\in K:~\|x\|\leq n\}$ and $K_n^\circ=\{x\in K:~\|x\|<n\}$.

\begin{proposition}\label{inclusion-EP2}
Suppose that for every $x,y_1,y_2\in K$, the following implication holds:
\begin{equation}\label{alfa}
[~f(x,y_1)\leq 0 \mbox{ and } f(x,y_2)<0~] \Rightarrow~ f(x,y_t)<0,~\forall t\in]0,1[,
\end{equation}
where $y_t=ty_1+(1-t)y_2$. If for some $n\in\N$ and some $x\in \operatorname{EP}(f,K_n)$ there exists $y\in K_n^\circ$ such that $f(x,y)\leq0$, then
$x\in \operatorname{EP}(f,K)$.
\end{proposition}

\begin{proof}
Let $x\in \operatorname{EP}(f,K_n)$ and $w\in K\setminus K_n$, if $f(x,w)<0$ then by (\ref{alfa})
$f(x,y_t)<0$ for all $t\in]0,1[$, where $y_t=ty+(1-t)w$. On the other hand,
since $y\in K_n^\circ$ there exists $t_0\in]0,1[$ such that $y_{t_0}\in K_n$,
which is a contradiction.
\end{proof}
{\begin{remark}
Condition (\ref{alfa}) is a technical assumption introduced by Farajzadeh and Zafarani in \cite{Fa-Za}  in
order to show the inclusion of $\operatorname{CFP}_{\operatorname{local}}(f,K)$ in $\operatorname{EP}(f,K)$. Clearly, the 
semistrict quasiconvexity of $f(x,\cdot)$
guarantees the condition (\ref{alfa}). So, in the Proposition \ref{inclusion-EP2} we can change
condition \eqref{alfa} by $f\in\mathcal{SQ}(K)$ and use the Remark \ref{regu-semi-equivalente} to
guarantee the nonemptiness of $\operatorname{EP}(f,K)$.
\end{remark}}


{As a direct consequence of previous result we have the following corollary.}

\begin{cor}\label{inclu-3}
Suppose (\ref{alfa}) holds  and $f$ has the upper sign property {(or $f\in\mathcal{SQ}(K)$ and $f_q$ has the upper sign property)}.
If for some $n\in\N$ and some $x\in \operatorname{CFP}(f,K_n)$ there exists $y\in K_n^\circ$ such that $f(x,y)\leq0$, then
$x\in \operatorname{EP}(f,K)$.  
\end{cor}

\begin{proof}
Is a direct consequence of Proposition \ref{inclusion-EP2} and Proposition \ref{caste-ACI} \emph{(ii)}.
\end{proof}




\begin{remark}
{The Corollary \ref{inclu-3} is an extension} of \cite[Lemma 4.1]{IKS06}.
\end{remark}

The following \emph{coercivity conditions} were studied in \cite{IS03,IKS09} and \cite{castellani2012}:

\begin{enumerate}[label=(C\arabic*), ref=(C\arabic*)]
\item \label{condicion1} For every sequence $\{x_n\}\subset K\setminus\{0\}$ satisfying
 $\displaystyle\lim_{n\to+\infty}\|x_n\|=\infty$, there exists  $u\in K$ and $n_0\in\N$ such that
  $f(x_n,u)\leq0$ for all $n\geq n_0$.
\item \label{condicion2} For every sequence $\{x_n\}\subset K\setminus\{0\}$ satisfying
 $\displaystyle\lim_{n\to+\infty}\|x_n\|=\infty$, there exists $n_0\in\N$ and $u_{n_0}\in K$  such that
 $\|u_{n_0}\|<\|x_{n_0}\|$ and  $f(x_{n_0},u_{n_0})\leq0$. 
\item \label{condicion3} For every sequence $\{x_n\}\subset K\setminus\{0\}$ such that
  $\displaystyle\lim_{n\to\infty}\|x_n\|=\infty$ and such that the sequence $\{\|x_n\|^{-1}x_n\}$ converges weakly to a point $x\in X$ 
  such that $y+x\in K$ and $f(y,x+y)\leq0$ for all $y\in K$,
 there exists another sequence $\{u_n\}\subset K$ such that, for $n$ large enough, $\|u_n\|<\|x_n\|$ and $f(x_n,u_n)\leq0$.
\end{enumerate}

It is not difficult to verify that \ref{condicion1} implies \ref{condicion2}, which in turn implies \ref{condicion3}. \\
 \newline
Clearly, if $f\in{\overline{\mathcal{C}}(K)}$ then (\ref{desi1}) implies that $f_i$ satisfies the coercivity conditions \ref{condicion1} or \ref{condicion2}
for all $i\in\{s,c,\overline{c},\overline{q},q\}$ provided that $f$ satisfies the same condition too.\\
\newline
{We define the following subfamily of $\mathcal{Q}(K)$:
\[
\overline{ \mathcal{SQ}}(K)=\{f\in \overline{\mathcal{Q}}(K):~ f_{\overline{q}}(x,\cdot)\mbox{ is semistrictly quasiconvex for all }x\in K\}
\]
Clearly, $\overline{\mathcal{C}}(K)\subset  \overline{\mathcal{SQ}}(K)\subset  \overline{\mathcal{Q}}(K)$.}\\
\newline
The following result extends the sufficient part of \cite[Theorem 4.4 (i)]{IKS09}, and 
also \cite[Theorem 5]{castellani2012}
with $\mu=0$.

\begin{proposition}\label{C3-existence}
 Suppose $X$ is a reflexive Banach space and $K$ is closed convex. If $f\in {\overline{\mathcal{SQ}}(K)}$ is such that
$f_{\overline{q}}$ is quasimonotone, it has the upper sign property on $K$ and it satisfies the  coercivity condition \ref{condicion3}, then $\operatorname{EP}(f,K)$ is nonempty.
\end{proposition}

\begin{proof}
If $f_{\overline{q}}$ is not properly quasimonotone, then by {\cite[Theorem 3 and Corollary 1]{castellani2012}} $\operatorname{EP}(f_{\overline{q}},K)$
is nonempty and the result follows from Remark \ref{regu-semi-equivalente}. Now, suppose that $f_{\overline{q}}$ is properly quasimonotone.
Since $K_n$ is a weakly compact set, Corollary \ref{CFP} implies that $\operatorname{CFP}(f_{\overline{q}},K_n)$ is nonempty. 
If there exists $x_n\in \operatorname{CFP}(f_{\overline{q}},K_n)$ such that $\|x_n\|<n$ then Corollary \ref{inclu-3} implies that $x_n\in \operatorname{EP}(f,K)$. Thus, we may assume that $\|x_n\|=n$ for all $n\in\N$. Since the unit ball of $X$ is
weakly compact, without loss of generality we may assume that $ \{x_n/n\}$ converges weakly to some $x\in X$. Fix $y\in K$ and $m>\|y\|$.
For $n\geq m$, $y\in K_n$. Since $x_n\in \operatorname{CFP}(f_{\overline{q}},K_n)$ we have that  
\[
 f_{\overline{q}}(y,x_n)\leq0.
\]
Let $z_n= (1/n)x_n+(1-1/n)y\in K_n$. Then 
\[
 f_{\overline{q}}(y,z_n)\leq0
\]
Clearly, $\{z_n\}$ converges weakly to $x+y\in K$. Hence, the lower semicontinuity of $f_{\overline{q}}(y,\cdot)$ implies that
\[
 f_{\overline{q}}(y,x+y)\leq0.
\]
Therefore, coercivity condition \ref{condicion3} implies that there exists a sequence $\{u_n\}\subset K$ such that $\|u_n\|<\|x_n\|$ and $f(x_n,u_n)\leq0$.
From Corollary \ref{inclu-3} we have that $\operatorname{EP}(f_{\overline{q}},K)$ is nonempty. The result follows from Remark \ref{regu-semi-equivalente}.
\end{proof}

\begin{proposition}\label{C2-existence}
 Suppose $X$ is a finite dimensional space and $K$ is closed convex subset of $X$. If $f\in \mathcal{SQ}(K)$ is
 such that ${f_{q}}(\cdot,y)$ is upper semicontinuous for all $y\in K$,  ${f_{q}}(x,x)=0$ for all $x\in K$, 
 and ${f_{q}}$ satisfies the  coercivity condition \ref{condicion2}, then $\operatorname{EP}(f,K)$ is nonempty.
\end{proposition}

\begin{proof}
Since $K_n$ is a compact set, then Corollary \ref{R1} implies that $\operatorname{EP}({f_{q}},K_n)$ is nonempty.  
If there exists $n\in\N$ such that $\|x_n\|<n$, then {Proposition \ref{inclusion-EP2}} with $x=y=x_n$, implies that
$x_n\in \operatorname{EP}({f_q},K)$ {and the result follows from Remark \ref{regu-semi-equivalente}}. 
If $\|x_n\|= n$ for all $n\in\N$, 
 condition \ref{condicion2} implies that there exists $n_0\in\N$ and $u\in K$ such that $u\in K_{n_0}^\circ$ and ${f_q}(x_{n_0},u)\leq0$. 
Using {Proposition \ref{inclusion-EP2}}  with $x=x_{n_0}$ and $y=u$, we have that $x_{n_0}\in \operatorname{EP}({f_q},K)$
{and the result follows again from Remark \ref{regu-semi-equivalente}}.
\end{proof}

\end{document}